\documentclass[12pt]{amsart}

\usepackage[colorlinks=true,linkcolor=blue]{hyperref}
\usepackage{amsmath}
\usepackage{amssymb}
\usepackage{amsthm}

\usepackage{graphicx}
\usepackage{epstopdf}
\usepackage{epsfig}

\usepackage{enumitem}
\usepackage{hyperref}

\usepackage{times}
\usepackage{relsize}
\usepackage{textcomp}
\usepackage{amssymb}
\usepackage[english]{babel}
\usepackage[autostyle]{csquotes}
\usepackage{epstopdf}
\usepackage{nicefrac}
\usepackage{tikz}

\theoremstyle{plain} 
\newtheorem{thm}{Theorem}[section]
\newtheorem{prop}[thm]{Proposition}
\newtheorem{lemma}[thm]{Lemma}
\newtheorem{cor}[thm]{Corollary} 
\newtheorem{question}[thm]{Question}

\theoremstyle{remark}
\newtheorem{remark}[thm]{Remark}

\newtheorem{notation}[thm]{Notation}
\theoremstyle{definition}
\newtheorem{defin}[thm]{Definition}

\newcommand{\CC}{\mathbb{C}}

\newcommand{\QQ}{\mathbb{Q}}

\newcommand{\ZZ}{\mathbb{Z}}

\DeclareMathOperator{\tr}{tr}

\DeclareMathOperator{\resid}{resid}
\DeclareMathOperator{\Res}{Res}

\newcommand{\dsps}{\displaystyle}

\begin{document}
	
	\title{The non-unit conjecture for Misiurewicz parameters}
	\date{June 15, 2025}
	\author[Benedetto]{Robert L. Benedetto}
	\address{Amherst College \\ Amherst, MA 01002 \\ USA}
	\email{rlbenedetto@amherst.edu}
	\author[Goksel]{Vefa Goksel}
	\address{Towson University \\ Towson, MD 21252 \\ USA}
	\email{vgoksel@towson.edu}
	
\begin{abstract}
A Misiurewicz parameter is a complex number $c$ for which
the orbit of the critical point $z=0$ under $z^2+c$ is strictly preperiodic.
Such parameters play the same role as special points in dynamical moduli spaces
that singular moduli (corresponding to CM elliptic curves) play as special points on modular curves.
Building on our earlier work, we investigate
whether the difference of two Misiurewicz parameters can be an algebraic unit.
%In a recent breakthrough, Li proved that the difference of two singular moduli, the $j$-invariants
%of CM elliptic curves, cannot be an algebraic unit.
(The corresponding question for singular moduli was recently answered in the negative by Li.)
We answer this dynamical question in many new cases under a widely believed irreducibility assumption. 
\end{abstract}
	
\maketitle

\section{Introduction}
Let $f\in \CC[z]$ be a polynomial of degree $d\geq 2$. We denote by $f^n$ the \emph{$n$-th iterate} of $f$, which is defined recursively by $f^n=f\circ f^{n-1}$ for $n\geq 2$.
We say that $f$ is called \emph{post-critically finite} (PCF) if all its critical points $\gamma$
have finite orbit $\mathcal{O}_f(\gamma):=\{f^n(\gamma) \, | \, n\geq 1\}$ under iteration of $f$.
In this work, we will focus on the PCF parameters $c$ in the family of quadratic polynomials
\[ f_{c}:=z^2+c\in \CC[z], \]
which has been of particular interest in the study of both complex and arithmetic dynamics.
Any quadratic polynomial over $\CC$ is linearly conjugate to a unique such $f_c$,
so that the family $\{f_c \, | \, c\in\CC\}$ is the appropriate dynamical parameter space of quadratic polynomials.

Since $0$ is the only (finite) critical orbit of $f_c$, the post-critical orbit of $f_c$ is given by
\[\mathcal{O}_{f_c}(0) = \{f_c(0),f_c^2(0),\dots\}.\]
If $x\in\CC$ satisfies $f_c^n(x) = x$ for some minimal positive integer $n$,
we say $x$ is \emph{periodic} (of exact period $n$). If
$x$ is not periodic but $f_c^m(x) = f_c^{m+n}(x)$ for some minimal positive integers $m$ and $n$,
we say $x$ is (strictly) \emph{preperiodic of type $(m,n)$}.
Note that for any PCF quadratic polynomial $f_c$, the point $0$ must be either periodic
or strictly preperiodic; and in the strictly preperiodic case, we must have $m\geq 2$.
For ease of notation, we write
\[ a_i =a_i(c) := f_c^i(0) \quad\text{for all } i\geq 0 \]
throughout this paper.
For $m\geq 2$, $n\geq 1$, the parameters $c$ such that
$f_{c}$ has a preperiodic orbit of type $(m,n)$ are called the \emph{Misiurewicz parameters} of type $(m,n)$.
They are the roots of a monic polynomial $G_{m,n}\in\ZZ[c]$ with integer coefficients, defined by
\begin{equation}
	\label{eq:Gdef}
	G_{m,n}(c):= \prod_{k|n} \big( a_{m+k-1} + a_{m-1}\big)^{\mu(n/k)}
	\cdot \begin{cases}
		\prod_{k|n} \big( a_{k}\big)^{-\mu(n/k)}
		& \text{ if } n|m-1,
		\\
		1 & \text{ if } n \nmid m-1.
	\end{cases}
\end{equation}
See \cite{BG1,BG2,BEK19,Buff18,Gok19,Gok20,HT15,Patra25} for further background on the Misiurewicz polynomials $G_{m,n}$.
%Given their arithmetic nature, Misiurewicz parameters have also drawn the attention of number theorists.

In \cite{BD13}, Baker and DeMarco proposed a dynamical version of the André-Oort Conjecture by relating
PCF parameters (including Misiurewicz parameters) in dynamical moduli spaces to special André-Oort points 
found on Shimura varieties. In this framework, PCF parameters in these moduli spaces play a role similar to that of
CM points on modular curves,  and they are believed to share similar arithmetic characteristics.
For additional information on the dynamical André-Oort Conjecture, see \cite{GKNY17}.

It has been conjectured that $G_{m,n}$ is irreducible over $\mathbb{Q}$ for all $m\geq 2$, $n\geq 1$;
see, for example, \cite[Remark~3.5]{Milnor12} and \cite[Exercise~4.17]{Sil07}.
Goksel \cite{Gok19, Gok20}, and Buff, Epstein, and Koch \cite{BEK19}
have proven this conjecture in some special cases.
Although this irreducibility conjecture remains open, our interest
in this paper relates to the following related dynamical question.
\begin{question}
	\label{question:c1-c2}
Let $c_0$ and $c_1$ be Misiurewicz parameters of type $(m,n)$ and $(k,\ell)$, where $(m,n)\neq (k,\ell)$. Under what conditions is the difference $c_0-c_1$ an algebraic unit?
\end{question}
If $G_{k,\ell}$ is irreducible not only over $\QQ$, but also over $\QQ(c_0)$,
as computational evidence has suggested,
then Question~\ref{question:c1-c2} would be equivalent to the following question.
\begin{question}
	\label{question:G_{m,ell}(c_1)}
Let $m,k\geq 2$, and $n,\ell\geq 1$. Let $c_0$ be a Misiurewicz parameter of type $(m,n)$, and suppose that $(m,n)\neq (k,\ell)$. Under what conditions is $G_{k,\ell}(c_0)$ an algebraic unit?
\end{question}
We answered Question~\ref{question:G_{m,ell}(c_1)} in many cases in \cite{BG1,BG2};
the only remaining case is when $m=k$ and $\ell$ is a proper divisor of $n$.
In \cite[Conjecture 1.3]{BG2}, we conjectured that $G_{m,\ell}(c_0)$ is never an algebraic unit in this case,
and in \cite[Theorem 1.7]{BG2}, we proved this conjecture in the cases $\ell=1,2$,
under the assumption that $G_{m,n}$ is irreducible over $\QQ$.

Our main result in this paper generalizes \cite[Theorem 1.7]{BG2} to many new periods $\ell$, under the same irreducibility assumption.
\begin{thm}
\label{thm:Main result}
Let $m\geq 2$, and $p$ be a prime not exceeding $1024$. Let $n\neq p$ be a positive integer divisible by $p$.
%Let $c_1$ be a root of $G_{m,p}$ and $c_2$ a root of $G_{m,n}$.
Let $c_0$ be a root of $G_{m,n}$.
Suppose that $G_{m,p}$ and $G_{m,n}$ are irreducible over $\QQ$.
%Then $c_1-c_2$ is not an algebraic unit. 
Then $G_{m,p}(c_0)$ is not an algebraic unit. 
\end{thm}

%\begin{thm}
%\label{thm:all but finitely many m}
%Let $p$ be a prime, and let $n\neq p$ be a positive integer divisible by $p$. Then the following holds for all but finitely many $m\geq 2$: let $c_1$ be a root of $G_{m,p}$ and $c_2$ a root of $G_{m,n}$. Suppose that $G_{m,p}$ and $G_{m,n}$ are irreducible over $\QQ$. Then $c_1-c_2$ is not an algebraic unit. 
%\end{thm}

Theorem~\ref{thm:Main result} may be considered a dynamical analogue of a recent result of Li \cite{Li21}
on \emph{singular moduli}, which are $j$-invariants of elliptic curves with CM (complex multiplication).
In their seminal work \cite{GZ84}, Gross and Zagier found an explicit factorization of the algebraic norm
of differences of singular moduli, showing that they have many small prime divisors in many cases. The question of whether these differences can be algebraic units or not has been a long-standing conjecture in arithmetic geometry. In a recent breakthrough, Li \cite{Li21} proved that such differences can never be algebraic units,
and he used this result to generalize the effective results of Andr\'{e}-Oort type.
In the notation of Theorem~\ref{thm:Main result}, if $G_{m,p}$ is irreducible not only over $\QQ$,
but even over $\QQ(c_0)$, then by the resulting equivalence of Question~\ref{question:c1-c2} and
Question~\ref{question:G_{m,ell}(c_1)} mentioned above, Theorem~\ref{thm:Main result}
would yield the corresponding result for such Misiurewicz parameters: that the difference of a root
of $G_{m,n}$ and a root of $G_{m,p}$ cannot be an algebraic unit.

The outline of the paper is as follows. We introduce key terminology including multiplier polynomials and
$p$-specialness in Section~\ref{sec:prelim}, where we also prove several results
about sums of roots of certain polynomials.
(We introduced $p$-specialness in \cite{BG2} to study Question~\ref{question:c1-c2}, but in this
paper we push the concept further, making heavy use of the elementary complex analysis result
of Proposition~\ref{prop:contour}.)
Section~\ref{sec:tracerelate} is devoted to the statement and proof of Theorem~\ref{thm:P_{m,n}^2},
relating the coefficients of certain multiplier polynomials. The proof requires several technical lemmas,
which we continue to make use of Proposition~\ref{prop:contour}.
In Section~\ref{sec:mainproof}, we combine these ingredients to prove Theorem~\ref{thm:Main result}
by inductively showing that the relevant multiplier polynomials are $2$-special.
Finally, Section~\ref{sec:data} is an appendix of some computational data needed
to verify the base cases of Section~\ref{sec:mainproof}.

\section{Preliminaries}
\label{sec:prelim}
If $f\in\CC(z)$ is a rational function, and if $x\in\CC$ is periodic of exact period $n\geq 1$,
the \emph{multiplier} $\lambda\in \CC$ of $x$ is defined by  
\begin{equation}
	\label{eq:multdef}
	\lambda:=\big(f^n\big)'(x) = \prod_{i=0}^{n-1} f'\big(f^i(x)\big) .
\end{equation}
The multiplier $\lambda$ is the same for each point $x, f(x), f^2(x), \ldots, f^{n-1}(x)$
of the forward orbit of $x$, and hence we may refer to the multiplier of the periodic cycle.
Because the multiplier is invariant under conjugation by a linear fractional transformation,
one can also define the multiplier of a periodic point at $x=\infty$ via change of coordinates.

Let $m\geq 2$ and $n \geq 1$. Now consider a root $\alpha$ of $G_{m,n}$.
%, and set $K=\QQ(\alpha)$.
Since $f_{\alpha}(z)=z^2+\alpha$ has $f'(z)=2z$, 
and since $a_{m+n-1}(\alpha)+a_{m-1}(\alpha)=0$,
equation~\eqref{eq:multdef} shows that the multiplier
of the periodic cycle $\{a_m(\alpha),\dots, a_{m+n-1}(\alpha)\}$ is
\begin{equation}
\label{eq:lambda}
\lambda_{m,n}(\alpha) = 2^n \Bigg(\prod_{i=0}^{n-1}a_{m+i}(\alpha)\Bigg)
= -2^n \Bigg(\prod_{i=0}^{n-1}a_{m+i-1}(\alpha)\Bigg).
\end{equation}
We are thus led to define the following
\emph{multiplier polymomial}, which has the multipliers of type $(m,n)$ as roots.

\begin{defin}
\label{def:multiplier poly}
Let $m\geq 2$ and $n\geq 1$ be integers. Let $\alpha_1,\dots,\alpha_k$ be all the roots of $G_{m,n}$.
The \emph{multiplier polynomial} $P_{m,n}$ associated with $G_{m,n}$ is
\[P_{m,n}(x)= \prod_{j=1}^{k} (x-\lambda_{m,n}(\alpha_j)) \in \ZZ[x], \]
where $\lambda_{m,n}(\cdot)$ is defined as in equation~\eqref{eq:lambda}.
\end{defin}

We will make use of the following two results and one definition from \cite{BG2}.
Throughout this paper, for any prime $p$, we write $v_p$ for the $p$-adic valuation
on $\ZZ$, normalized so that $v_p(p)=1$.

\begin{prop}[\cite{BG2}, Proposition~5.3]
\label{prop:implies}
Let $m,n\geq 2$ be integers. Let $\alpha$ be a root of $G_{m,n}$.
Let $\ell\geq 1$ be a proper divisor of $n$, and suppose that $\Res(P_{m,\ell}, \Phi_{n/\ell})\neq \pm 1$.
If $G_{m,n}$ and $G_{m,\ell}$ are irreducible over $\QQ$, then $G_{m,\ell}(\alpha)$ is not an algebraic unit.
\end{prop}

\begin{defin}[\cite{BG2}, Definition~6.1]
\label{def:special}
Let $P(x)=x^i+A_{i-1}x^{i-1}+\dots+A_1x+A_0\in\ZZ[x]$ be a monic polynomial with integer coefficients.
Let $p$ be a prime number.
%, and let $v_p$ be the $p$-adic valuation on $\ZZ$.
We say that $P(x)$ is \emph{p-special} if it satisfies the following two properties:
\begin{itemize}
	\item $v_p(A_{i-1})> v_p(2)$, and
	\item $v_p(A_j)>v_p(A_{i-1})$ for $j=0,1,\dots,i-2$.
\end{itemize}
\end{defin}

\begin{thm}[\cite{BG2}, Theorem~6.2]
\label{thm:special}
Let $P(x)\in \ZZ[x]$ be a $p$-special polynomial for some prime $p$.
Then for every integer $\ell \geq 1$, we have $|\Res(P,\Phi_{\ell})|>1$,
where $\Phi_{\ell}(x)\in \ZZ[x]$ is the $\ell$-th cyclotomic polynomial.
\end{thm}

%The following corollary draws a connection between differences of Misiurewicz parameters
%and the $2$-specialness of the multiplier polynomials.
%It will be one of the main tools in the proof of Theorem~\ref{thm:Main result}.
%\begin{cor}
%\label{cor:main tool for non-unitness}
%Let $m,n\geq 2$ be integers. Let $\ell$ be a proper divisor of $n$, and suppose that $P_{m,\ell}$ is $2$-special. Assume that $G_{m,n}$ and $G_{m,\ell}$ are irreducible over $\QQ$. If $c_1$ is a root of $G_{m,n}$ and $c_2$ is a root of $G_{m,\ell}$, then $c_1-c_2$ is not an algebraic unit. 
%\end{cor}
%
%\begin{proof}
%Using Theorem~\ref{thm:special} for $p=2$, we have $\Res(P_{m,\ell},\Phi_i)\neq \pm 1$ for any $i\geq 1$. By Proposition~\ref{prop:implies}, then, we conclude that $G_{m,\ell}(c_1)$ is not an algebraic unit. Since $G_{m,n}$ and $G_{m,\ell}$ are both irreducible over $\QQ$ by assumption, Lemma~\ref{lemma:f(beta) unit} now implies that $c_1-c_2$ is not an algebraic unit, as desired. 
%\end{proof}

Our next result will be useful for proving that multiplier polynomials are $2$-special.
In particular, it shows that for $p=2$,
the first bullet point of Definition~\ref{def:special} holds for multiplier polynomials,
and it provides a useful lower bound for the valuations that arise in the second bullet point.

\begin{prop}
\label{prop:prespecial}
Let $m\geq 2$ and $n\geq 1$ be integers, and let $P_{m,n}(x)=x^k + \sum_{i=0}^{k-1} b_i x^i$
be the associated multiplier polynomial. Then for each $\ell =1,\ldots, k$, we have
\[ v_2(b_{k-\ell})\geq n\ell+1 . \]
\end{prop}

\begin{proof}
Let $\alpha_1,\ldots,\alpha_k$ be the roots of $G_{m,n}$. Since $G_{m,n}(c) \in\ZZ[c]$ is monic,
each $\alpha_i$ is an algebraic integer. Because $a_j(c)\in\ZZ[c]$ for each $j\geq 1$,
it follows that each $a_j(\alpha_i)$ is also an algebraic integer,
so that $v_2(a_j(\alpha_i))\geq 0$, for any extension of the valuation $v_2$ to $\QQ(\alpha_i)$.

Furthermore, for any $j$ divisible by $n$, we have $v_2(a_j(\alpha_i))>0$,
by \cite[Theorem~1.3]{Gok19} (or by \cite[Theorem~3.3]{BG1}).
Thus, the multiplier $\lambda_i=\lambda_{m,n}(\alpha_i)$ satisfies
\[ v_2(\lambda_i) = v_2\bigg( 2^n \prod_{j=0}^{n-1}a_{m+j-1}(\alpha_i)\bigg) > n v_2(2)\]
by equation~\eqref{eq:lambda}, together with the fact that $n|(m+j-1)$ for some $0\leq j\leq n-1$.
By the definition of $P_{m,n}$, it follows that $v_2(b_{k-\ell}) > n\ell v_2(2)$,
since $b_{k-\ell}$ is a homogeneous polynomial in $\lambda_1,\ldots,\lambda_k$ of degree $\ell$.
Finally, because $b_{k-\ell}\in\ZZ$, its $2$-adic valuation is an integer,
and therefore $v_2(b_{k-\ell})\geq n\ell+1$.
\end{proof}

The following definition will be handy in the computations throughout the paper. 
%Motivated by Proposition~\ref{prop:contour}, we make the following definition.

\begin{defin}
\label{def:T(f,g)}
Let $K$ be a field, and let $f,g\in K[x]$ with $g\neq 0$.
Define $T(f,g)\in K$ to be the sum of $f(c_i)$ over all the roots $c_i$ of $g$,
repeated according to multiplicity.
\end{defin}

The following lemma establishes some basic properties of $T(f,g)$.

\begin{lemma}
\label{lemma:properties of T}
Let $K$ be a field. Then:
\begin{enumerate}[label=(\alph*)]
	\item T is linear in the first coordinate, i.e.,
	\[ T(af,h)=aT(f,h) \quad \text{and}\quad T(f+g,h)=T(f,h)+T(g,h) \]
	for any $a\in K$ and $f,g,h\in K[x]$ with $h\neq 0$.
	\item $T(f,gh)=T(f,g)+T(f,h)$ for any $f,g,h\in K[x]$ with $g, h\neq 0$.
\end{enumerate}
\end{lemma}

\begin{proof}
Immediate from the definition.
\end{proof}

Our next result, an elementary fact from complex analysis, will be essential in many proofs in this section.

\begin{prop}
	\label{prop:contour}
	Let $f,g \in \CC[x]$ be polynomials such that $f(0) = g(0) = 0$, and $g'(0)\neq 0$.
	Let $0=c_1,c_2,\dots, c_k$ be the roots of $g$, and assume that none are multiple roots.
	Let $h\in \CC[x]$ be the remainder when $fg'$ is divided by the polynomial $g(x)/x$,
	i.e., $\deg h < (\deg g) - 1$ and $h\equiv fg' \pmod{g/x}$. Then
%	\[ \sum_{i=1}^{k} f(c_i) = -\frac{h(0)}{g'(0)}. \]
	\[ T(f,g) = -\frac{h(0)}{g'(0)}. \]
\end{prop}
\begin{proof}
	For any large enough radius $R>0$, we know from complex analysis that
	\begin{equation}
		\label{eq:resfmla}
		\frac{1}{2\pi i} \oint_{|z|=R} \frac{h(z)\, dz}{g(z)} = \sum_{i=1}^k \resid\bigg(\frac{h}{g} , c_i \bigg),
	\end{equation}
	where $\resid(F,c)$ denotes the residue of $F(z)$ at $z=c$.
	Since $g$ has only simple roots, we have $\resid(h/g,c_i) = h(c_i)/g'(c_i)$.
	On the other hand, for sufficiently large $|z|$, we have $|h(z)/g(z)| = O(|z|^{-2})$,
	because $\deg g - \deg h \geq 2$.
	Thus, taking the limit of equation~\eqref{eq:resfmla} as $R\to\infty$, we obtain
	\begin{equation}
		\label{eq:11}
		\sum_{i=1}^{k} \frac{h(c_i)}{g'(c_i)} = 0.
	\end{equation}
	
	By our choice of $h$, we have $h(c_i) / g'(c_i) = f(c_i)$ for $i=2,3,\dots,k$. Hence, we obtain
	\[ \sum_{i=2}^{k} \frac{h(c_i)}{g'(c_i)} = \sum_{i=2}^{k} f(c_i). \]
	Because $f(c_1) = f(0) = 0$ and using (\ref{eq:11}), the desired identity follows immediately.
	%\[ \sum_{i=1}^{k} f(c_i) = -\frac{h(0)}{g'(0)} \]
\end{proof}

Recall that the trace $\tr(P)$ of a monic polynomial $P(x)$ of degree $n\geq 1$
is the sum of its roots, or equivalently, the negative of the $x^{n-1}$-coefficient of $P$.
Our next result provides an explicit formula for the trace of the multiplier polynomial $P_{m,n}$.
%Recall that the $2$-adic valuation of the trace of the multiplier polynomial plays a key role in $2$-specialness.

\begin{lemma}
\label{lemma:Mobius sum for trace}
Let $m\geq 2$, $n\geq 1$. Then
\[\tr(P_{m,n}) = 2^n\sum_{d|n} (-1)^{n/d}
\mu\bigg(\frac{n}{d}\bigg)T\Bigg(\bigg(\prod_{i=m-1}^{m+d-2}a_i\bigg)^{n/d}, a_{m+d-1}+a_{m-1}\Bigg).\]

\end{lemma}
\begin{proof}
First assume that $n\nmid m-1$. Then we have
\begin{align*}
\tr(P_{m,n}) &= T\bigg(2^n\prod_{i=m}^{m+n-1}a_i, G_{m,n}\bigg)\\
&= 2^n T\bigg(\prod_{i=m}^{m+n-1}a_i, \prod_{d|n} (a_{m+d-1}+a_{m-1})^{\mu(n/d)}\bigg)\\
&= 2^n\sum_{d|n}\mu\bigg(\frac{n}{d}\bigg)T\bigg(\prod_{i=m}^{m+n-1}a_i, a_{m+d-1}+a_{m-1}\bigg)\\
&= 2^n\sum_{d|n} (-1)^{n/d}\mu\bigg(\frac{n}{d}\bigg)
T\Bigg(\bigg(\prod_{i=m-1}^{m+d-2}a_i\bigg)^{n/d}, a_{m+d-1}+a_{m-1}\Bigg),
\end{align*}
where we used Lemma~\ref{lemma:properties of T}(a) in the second equality and Lemma~\ref{lemma:properties of T}(b) in the third. The fourth equality also follows by periodicity and the fact that for any root $c$ of $a_{m+d-1}+a_{m-1}$, we have $a_{m+d-1}(c)=-a_{m-1}(c)$. Thus, we are done in this case.\par

We now assume $n | m-1$. By the first part of the proof and using Lemma~\ref{lemma:properties of T}(b), we have
\begin{align*}
	\tr(P_{m,n}) &= T\bigg(2^n\prod_{i=m}^{m+n-1}a_i, G_{m,n}\bigg)\\
	&= 2^n \Bigg(T\bigg(\prod_{i=m}^{m+n-1}a_i, \prod_{d|n} (a_{m+d-1}+a_{m-1})^{\mu(n/d)}\bigg)
	-T\bigg(\prod_{i=m}^{m+n-1}a_i, \prod_{d|n} a_d^{\mu(n/d)}\bigg)\Bigg)\\
	&= 2^n\sum_{d|n} \mu\bigg(\frac{n}{d}\bigg)
	\Bigg(T\bigg(\prod_{i=m}^{m+n-1}a_i, a_{m+d-1}+a_{m-1}\bigg)-T\bigg(\prod_{i=m}^{m+n-1}a_i, a_d\bigg)\Bigg)\\
	&= 2^n\sum_{d|n} (-1)^{n/d}\mu\bigg(\frac{n}{d}\bigg)
	T\Bigg(\bigg(\prod_{i=m-1}^{m+d-2}a_i\bigg)^{n/d}, a_{m+d-1}+a_{m-1}\Bigg).
\end{align*}
Note that the fourth equality follows because for any root $c$ of $a_d$, by periodicity, $a_{dk}(c)=0$ for any $k\geq 1$, thus $a_i(c)=0$ for some $i\in \{m, m+1, \dots, m+n-1\}$.
%The proof of Lemma~\ref{lemma:Mobius sum for trace} is now complete.
\end{proof}
Since the case that $n=p$ is a prime is of particular interest for our main result, we now use Lemma~\ref{lemma:Mobius sum for trace} to describe $\tr(P_{m,p})$ for any odd prime $p$.

\begin{thm}
\label{thm:Trace(P_{m,p})}
%Let $m\geq 2$, and let $p\geq 3$ be an odd prime number. Then
Let $m\geq 3$, and let $p\geq 3$ be an odd prime number. Then
\[\tr(P_{m,p}) =
%\begin{cases} -2^{p+1} & \text{if } m=2 \\
%2^p\big(T(a_{m-1}^p, a_m+a_{m-1})+2^{m+p-2}-2^{m-2}\big)& \text{if }m>2. \end{cases}\]
2^p\big(T(a_{m-1}^p, a_m+a_{m-1})+2^{m+p-2}-2^{m-2}\big) .\]
\end{thm}

\begin{proof}
By Lemma~\ref{lemma:Mobius sum for trace}, we have
\begin{align}
\label{eq:tracePmp}
\tr(P_{m,p}) &= 2^p\sum_{d|p} (-1)^{p/d}\mu\bigg(\frac{p}{d}\bigg)
T\Bigg(\bigg(\prod_{i=m-1}^{m+d-2}a_i\bigg)^{p/d}, a_{m+d-1}+a_{m-1}\Bigg)
\notag \\
&= 2^p\Bigg( T\big(a_{m-1}^p, a_{m}+a_{m-1}\big)
- T\bigg(\prod_{i=m-1}^{m+p-2}a_i, a_{m+p-1}+a_{m-1}\bigg)\Bigg).
\end{align}
In the rest of the proof, for any polynomial $g\in \ZZ[c]$, we denote by $\overline{g}\in \ZZ[c]$
the remainder of $g$ when divided  by $(a_{m+p-1}+a_{m-1})/c$.
In particular, if $\deg(g)< \deg( (a_{m+p-1}+a_{m-1})/c )$, then $g=\overline{g}$.

By Proposition~\ref{prop:contour}, we have
\[T\bigg(\prod_{i=m-1}^{m+p-2} a_i, a_{m+p-1}+a_{m-1}\bigg) = -\frac{\overline{H}(0)}{2},\]
where $H=(\prod_{i=m-1}^{m+p-2} a_i)(a_{m+p-1}'+a_{m-1}')$.
%when divided by $(a_{m+p-1}+a_{m-1})/c$.
Write $H=A+B$, where
\[A = a_{m+p-1}' \prod_{i=m-1}^{m+p-2} a_i
\quad\text{and}\quad
B = a_{m-1}' \prod_{i=m-1}^{m+p-2} a_i.\]
To compute $\overline{H}$,
we will now separately calculate the remainders $\overline{A}$ and $\overline{B}$
of $A$ and $B$ when divided by $(a_{m+p-1}+a_{m-1})/c$.

\smallskip

\textbf{Remainder of $\boldsymbol{A}$.}
Since $a_{m+p-1} = a_{m+p-2}^2+c$, we have $a_{m+p-1}' = 2a_{m+p-2}a_{m+p-2}'+1$,
and hence 
\[ A = A_1 + A_2,
\quad\text{where}\quad
A_1 = 2a_{m+p-2}a_{m+p-2}' \prod_{i=m-1}^{m+p-2} a_i
\quad\text{and}\quad
A_2 = \prod_{i=m-1}^{m+p-2} a_i . \]
We first consider $A_2$. Note that
\begin{equation}
\label{eq:degA2}
\deg(A_2) = \sum_{i=m-1}^{m+p-2}2^{i-1} = 2^{m+p-2}-2^{m-2}.
\end{equation}
Since $m\geq 3$, we have $\deg(A_2)<2^{m+p-2}-1 = \deg((a_{m+p-1}+a_{m-1})/c)$, and hence
$\overline{A_2}=A_2$, which vanishes at $0$.

We now consider $A_1$. Noting that
\[ a_{m+p-2}^2 = a_{m+p-1} -c \equiv -a_{m-1} - c \pmod{(a_{m+p-1}+a_{m-1})/c}, \]
we obtain
\[ A_1 \equiv 2a_{m+p-2}'(-a_{m-1}-c)\prod_{i=m-1}^{m+p-3}a_i = -C_1 - C_2 \pmod{(a_{m+p-1}+a_{m-1})/c}, \]
where 
\[C_1 = 2a_{m+p-2}'a_{m-1}\prod_{i=m-1}^{m+p-3}a_i
\quad\text{and}\quad
C_2 = 2ca_{m+p-2}'\prod_{i=m-1}^{m+p-3}a_i .\]
By direct computation, we obtain
\[\deg(C_1) = 2^{m+p-3}-1+ 2^{m-2} +\sum_{i=m-1}^{m+p-3}2^{i-1} = 2^{m+p-2}-1
= \deg((a_{m+p-1}+a_{m-1})/c).\]
Thus, by comparing the leading coefficients, we have
\[\overline{C}_1 = C_1 - 2^{m+p-2}\bigg(\frac{a_{m+p-1}+a_{m-1}}{c}\bigg),\]
which evaluates to $-2^{m+p-1}$ at $0$.

We also have
\begin{equation}
\label{eq:degC2}
\deg(C_2) = 1+2^{m+p-3}-1+\sum_{i=m-1}^{m+p-3} 2^{i-1} = 2^{m+p-2}-2^{m-2}.
\end{equation}
As with $A_2$ above, we have
$\deg(C_2) < \deg((a_{m+p-1}+a_{m-1})/c)$, since $m\geq 3$.
Thus, $\overline{C}_2 = C_2$, which vanishes at $0$.

Combining the foregoing computations, we conclude that
%At the end of this part, using Eq.~\ref{eq:A=A_1+A_2} and Eq.~\ref{eq:A_1=C_1-C_2}, we conclude the following:
\begin{equation}
\label{eq:Abar}
\overline{A}(0) = -\overline{C}_1(0) - \overline{C}_2(0) + \overline{A}_2(0) = 2^{m+p-1}.
%\begin{cases} 2^{p+1}-2 & \text{if } m=2 \\ 	2^{m+p-1} & \text{if }m>2. \end{cases}
%	2^{m+p-1} & \text{if }m>2.
\end{equation}

\smallskip

\textbf{Remainder of $\boldsymbol{B}$.}
We have
\[\deg(B) = 2^{m-2}-1 + \sum_{i=m-1}^{m+p-2}2^{i-1} = 2^{m+p-2}-1 = \deg( (a_{m+p-1}+a_{m-1})/c).\]
Therefore, by comparing the leading coefficients, we obtain
\[\overline{B} = B - 2^{m-2}\bigg(\frac{a_{m+p-1}+a_{m-1}}{c}\bigg),\]
which evaluates to $-2^{m-1}$ at $0$. Thus, we conclude that
\begin{equation}
	\label{eq:Bbar}
	\overline{B}(0)=-2^{m-1}.
\end{equation}

Applying Proposition~\ref{prop:contour} to $H=A+B$,
together with equations~\eqref{eq:Abar} and~\eqref{eq:Bbar}, yields
\[T\bigg(\prod_{i=m-1}^{m+p-2}a_i, a_{m+p-1+a_{m-1}}\bigg)
= - \frac{H(0)}{2} 
= -\frac{1}{2}\Big( \overline{A}(0) + \overline{B}(0) \Big) =
%\begin{cases} 2-2^p & \text{if } m=2 \\ 2^{m-2}-2^{m+p-2} & \text{if }m>2. \end{cases}
2^{m-2}-2^{m+p-2}. \]
Plugging this value into equation~\eqref{eq:tracePmp} completes the proof.
\end{proof}

\begin{remark}
\label{remark:m2case}
We assumed $m\geq 3$ in Theorem~\ref{thm:Trace(P_{m,p})},
but the computations are only modestly more involved when $m=2$.
In particular, in that case, equations~\eqref{eq:degA2} and~\eqref{eq:degC2} yield
\[ \deg(A_2) =\deg(C_2) = \deg((a_{p+1}+a_{1})/c) .\]
Therefore, comparing the leading coefficients, we obtain
\[\overline{A}_2 = A_2 - \frac{a_{p+1}+a_{1}}{c} \quad\text{and}\quad
\overline{C}_2 = C_2 - 2^p \bigg( \frac{a_{p+1}+a_{1}}{c} \bigg) .\]
Because $(a_{p+1}+a_{1})/c$ evaluates to $2$ at $c=0$, it follows that
\[ \overline{A}_2(0)= -2 \quad\text{and}\quad \overline{C}_2(0) = -2^{p+1}. \]
We still have $\overline{C}_1(0)=-2^{p+1}$ and $\overline{B}(0)=-2$
from the proof of Theorem~\ref{thm:Trace(P_{m,p})},
and therefore
\[T\bigg(\prod_{i=1}^{p}a_i, a_{p+1}+a_{1} \bigg)
= -\frac{1}{2}\big(2^{p+1} + 2^{p+1} - 2 -2 \big) = 2-2^{p+1}.\]
We also have $a_2+a_1=c^2+2c$, and hence the roots of $a_2+a_1$
are simply $c_1=0$ and $c_2=-2$. Thus,
\[T(a_1^p, a_2+a_1) = T(c^p, a_2+a_1) = 0^2 + (-2)^p = -2^p, \]
by direct computation. Equation~\eqref{eq:tracePmp} then gives us
\[ \tr(P_{2,p}) = 2^p \big( -2^p - (2-2^{p+1})\big) = 2^p(2^p-2) = 2^{2p} - 2^{p+1} . \]
\end{remark}

\begin{remark}
\label{remark:trace in known extension}
In the proof of Theorem~\ref{thm:Trace(P_{m,p})} and in Remark~\ref{remark:m2case},
we never used the fact that $p$ is prime when calculating the expression
\[T\bigg(\prod_{i=m-1}^{m+p-2}a_i, a_{m+p-1}+a_{m-1} \bigg).\]
In particular, then, we have the formula
\[T\bigg(\prod_{i=m-1}^{m+n-2}a_i, a_{m+n-1}+a_{m-1} \bigg) =
\begin{cases} 
	2-2^{n+1} & \text{if } m=2 \\
	2^{m-2}-2^{m+n-2} & \text{if }m\geq 3.
\end{cases} \]
for any integers $m\geq 2$ and $n\geq 1$.
\end{remark}

\section{A trace relation for multiplier polynomials}
\label{sec:tracerelate}

The following result relates the coefficients of the multiplier polynomials $P_{m,n}$
under certain conditions on $m$ and $n$. Combined with \cite[Lemma 6.3]{BG2},
it will allow us to establish the $2$-specialness of the polynomials $P_{m,p}$
when $p$ is a prime less than $1024$.

\begin{thm}
\label{thm:P_{m,n}^2}
Let $m\geq 2$ and $n\geq 1$.
Write
\[ \big( P_{m,n}(x) \big)^2=x^k+\sum_{i=0}^{k-1} b_i x^i
\quad\text{and}\quad P_{m+1,n}(x) = x^{\ell}+\sum_{i=0}^{\ell-1} c_i x^i , \]
where $k=2\deg(P_{m,n})=2\deg(G_{m,n})$ and $\ell=\deg(P_{m+1,n})=\deg(G_{m+1,n})$.
For each integer $i\geq 1$ such that $\lfloor \frac{ni}{2}\rfloor \leq 2^{m-2}-2$, we have $b_{k-i}=c_{\ell-i}$.
\end{thm}

\begin{cor}
\label{cor:trace-doubles}
Let $m\geq 2$ and $n\geq 1$. Suppose that $\lfloor \frac{n}{2}\rfloor\leq2^{m-2}-2$.
Then $\tr(P_{m+1,n}) = 2\tr(P_{m,n})$.
\end{cor}

\begin{proof}
Since we have $\tr(P_{m,n}^2)=2\tr(P_{m,n})$,
Corollary~\eqref{cor:trace-doubles} immediately follows from Theorem~\ref{thm:P_{m,n}^2} with $i=1$.
\end{proof}

The rest of this section is devoted to the proof of Theorem~\ref{thm:P_{m,n}^2}.
We need several lemmas, some of which we discovered through empirical observations in Magma.

\begin{notation}
In this section, whenever we have integers $m\geq 2$ and $n\geq 1$, we fix the following notation.
For any polynomial $f\in \ZZ[c]$, we denote by $\overline{f}\in \ZZ[c]$
the remainder of $f$ when divided  by $(a_{m+n-1}+a_{m-1})/c$.
\end{notation}

\begin{lemma}
\label{lemma:simpleprod}
Fix positive integers $\ell, s\geq 1$ with $\ell \leq s+1$. Then
\[ a_{\ell} \prod_{j=\ell}^{s} a_j = a_{s+1} - c \sum_{i=\ell+1}^{s+1} \prod_{j=i}^{s} a_j , \]
where we understand an empty sum to be $0$, and an empty product to be $1$.
\end{lemma}

\begin{proof}
Fix $s$, and proceed by decreasing induction on $\ell$.

For $\ell=s+1$, both sides are simply $a_{s+1}$, verifying the equality.
Assuming the claim for $\ell + 1$, we have
\begin{align*}
a_{\ell} \prod_{j=\ell}^{s} a_j
&= a_{\ell}^2 \prod_{j=\ell+1}^{s} a_j
= (a_{\ell+1}-c) \prod_{j=\ell+1}^{s} a_j
= a_{\ell+1} \bigg(\prod_{j=\ell+1}^{s} a_j\bigg) -c \bigg(\prod_{j=\ell+1}^{s} a_j\bigg) \\
& = a_{s+1} - c \bigg( \sum_{i=\ell+2}^{s+1} \prod_{j=i}^{s} a_j \bigg)
- c \bigg( \prod_{j=\ell+1}^{s} a_j \bigg)
= a_{s+1} - c \sum_{i=\ell+1}^{s+1} \prod_{j=i}^{s} a_j ,
\end{align*}
where the second equality is because $a_{\ell+1}=a_{\ell}^2+c$,
and the fourth is by our inductive assumption.
\end{proof}

\begin{lemma}
\label{lemma:prodzero}
Fix an integer $1\leq \ell \leq m+n-1$, and let $R\in\ZZ[c]$. Define
\[ H = R a_\ell \bigg( \prod_{j=\ell}^{m+n-2} a_j \bigg) \in \ZZ[c] .\]
 If $\deg R \leq \min\{ 2^{m+n-2} - 2^{m-2} -2, 2^{\ell}-3\}$, then $\overline{H}(0)=0$.
\end{lemma}

\begin{proof}
Since $a_{m+n-1}\equiv - a_{m-1}$ modulo $(a_{m+n-1}+a_{m-1})/c$, 
it follows from Lemma~\ref{lemma:simpleprod} that $H$ is congruent to
\begin{equation}
\label{eq:Hwrite}
-R\big( a_{m-1} +  c Q \big) ,
\quad\text{where}\quad
Q=\sum_{i=\ell+1}^{m+n-1} \prod_{j=i}^{m+n-2} a_j.
\end{equation}
If $\ell=m+n-1$, then $Q=0$; otherwise,
the highest degree term in the sum defining $Q$ is the $i=\ell+1$ term, and hence
\[ \deg (cQ) = 1 + \sum_{j=\ell+1}^{m+n-2} 2^{j-1} = 1 + 2^{m+n-2} - 2^{\ell} . \]
Thus,
\begin{align*}
\deg\big( -R( a_{m-1} +  c Q ) \big)
& \leq \deg(R) + \max\big\{ 2^{m-2}, 1 + 2^{m+n-2} - 2^{\ell} \big\} \\
& \leq 2^{m+n-2}-2 < \deg\big( (a_{m+n-1} + a_{m-1}) / c\big) ,
\end{align*}
and hence $\overline{H} = -R(a_{m-1}+cQ)$. Since both $a_{m-1}$ and $c$ are zero at $c=0$,
it follows that $\overline{H}(0)=0$.
\end{proof}

\begin{lemma}
\label{lemma:multiply by c^k}
Let $\{i_1,i_2,\dots,i_t\}$ be a nonempty, proper subset of $\{m-1,m,\dots,m+n-2\}$ with $i_1<i_2<\dots<i_t$.
If $0\leq k\leq 2^{m-2}-2$, then
\[ T\bigg(c^k\prod_{j=1}^{t} a_{i_j}, a_{m+n-1}+a_{m-1}\bigg)= 0. \]
\end{lemma}

\begin{proof}
By Proposition~\ref{prop:contour}, we have 
\[ T\bigg(c^k\prod_{j=1}^{t} a_{i_j}, a_{m+n-1}+a_{m-1}\bigg)=-\frac{\overline{H}(0)}{2},\]
where $H = c^k (a_{m+n-1}'+a_{m-1}')\prod_{s=1}^{t} a_{i_s}$.
We will prove the lemma by showing that the polynomials
$\overline{H}_1$ and $\overline{H}_2$ both vanish at $c=0$, where
\[ H_1= c^k a_{m-1}' \prod_{s=1}^{t}a_{i_s}
\quad\text{and}\quad
H_2= c^k a_{m+n-1}' \prod_{s=1}^{t}a_{i_s} . \]

%We first consider $c^k(\prod_{j=1}^{t}a_{i_j} )a_{m-1}'$.
We have
\begin{align*}
	\deg(H_1)
%	\text{deg}(c^k(\prod_{j=1}^{t}a_{i_j} )a_{m-1}')
	&= k+ (2^{m-2}-1) + \sum_{s=1}^{t}2^{i_s-1}
	\leq k - 1 +\sum_{j=m-1}^{m+n-2}2^{j-1} \\
%	\leq (2^{m-2}-2) - 1 +\sum_{j=m-1}^{m+n-2}2^{j-1} \\
	& \leq 2^{m+n-2}-3 <  2^{m+n-2}-1 = \deg\big( (a_{m+n-1}+a_{m-1})/c \big),
\end{align*}
where in the first inequality we have invoked the fact that $\{i_1,\ldots,i_t\}$ is a \emph{proper}
subset of $\{m-1,\ldots,m+n-2\}$ to absorb the $2^{m-2}$ into the sum,
and in the second we have used the fact that $k\leq 2^{m-2}-2$.
It follows that $\overline{H}_1=H_1$,
%=c^k a_{m-1}' \prod_{s=1}^{t}a_{i_s}$, 
which vanishes at $c=0$ because $t\geq 1$.

We now consider $H_2$. Note that for any $r\geq 2$, we have
$a_{r}= c + a_{r-1}^2$, and hence
\begin{equation}
\label{eq:arprime}
a_{r}' = 1 + 2a_{r-1}a_{r-1}' .
\end{equation}
%for any $r\geq 0$.
Applying this formula inductively, starting from $a'_1=1$, we obtain
\[ a_{r}' = \sum_{i=0}^{r-1} 2^i \bigg(\prod_{j=r-i}^{r-1}a_j\bigg), \]
where we understand the empty product in the $i=0$ term to be $1$. Therefore,
\begin{equation}
\label{eq:H2sum}
H_2 = c^k a_{m+n-1}' \prod_{s=1}^{t}a_{i_s}
= \sum_{i=0}^{m+n-2} \Bigg( 2^i c^k \bigg(\prod_{s=1}^{t}a_{i_s}\bigg)
\bigg( \prod_{j=m+n-i-1}^{m+n-2}a_j\bigg) \Bigg).
\end{equation}
%We will show that the remainder of each term in this sum, when divided by $(a_{m+n-1}+a_{m-1})/c$,
%vanishes at $0$, which will finish the proof.
%
For each $0\leq i\leq m+n-2$, the $i$-th term of the sum in equation~\eqref{eq:H2sum} is
\[ H_{2,i} = 2^i c^k \bigg(\prod_{s=1}^{t}a_{i_s}\bigg) \bigg( \prod_{j=m+n-i-1}^{m+n-2}a_j\bigg) .\]
It suffices to show that for each such $i$, we have $\overline{H}_{2,i}(0)=0$.
We consider two cases.

%Next, consider an arbitrary term of the form
%$C_i=c^k(\prod_{s=1}^{t}a_{i_s})\prod_{j=m+n-i-1}^{m+n-2}a_j$.
First, suppose that $i_t< m+n-i-1$. Then because $m-1\leq i_1 < \cdots < i_t < m+n-i-1$, we have
	% $\{i_1,i_2,\dots,i_t\}\cap \{n-i,\dots, n-1\} = \emptyset$. Then
\begin{align*}
%	\deg \bigg( c^k (\prod_{s=1}^{t}a_{i_s})\prod_{j=m+n-i-1}^{m+n-2}a_j \bigg)
	\deg(H_{2,i})
	&= k + \sum_{s=1}^{t}2^{i_s-1} + \sum_{j=m+n-i-1}^{m+n-2} 2^{j-1}
	\leq  2^{m-2} - 2 + \sum_{j=m-1}^{m+n-2} 2^{j-1} \\
	&= 2^{m+n-2}-2
	< 2^{m+n-2}-1 = \deg\big( (a_{m+n-1}+a_{m-1})/c \big),
\end{align*}
	and as before it follows that $\overline{H}_{2,i} = H_{2,i}$, which vanishes at $0$.
%\[c^k(\prod_{s=1}^{t}a_{i_s})\prod_{j=m+n-i-1}^{m+n-2}a_j
%= \overline{c^k(\prod_{s=1}^{t}a_{i_s})\prod_{j=m+n-i-1}^{m+n-2}a_j},\]

Second, suppose that $i_t \geq m+n-i-1$. Let $\ell=i_t$, so that $m-1\leq \ell \leq m+n-2$.
Then we have
\[ H_{2,i} = R a_{\ell} \bigg(\prod_{j=\ell}^{m+n-2} a_j\bigg),
\quad\text{where}\quad
R =2^i  c^k\bigg(\prod_{s=1}^{t-1}a_{i_s}\bigg) \bigg(\prod_{j=m+n-i-1}^{\ell-1}a_j\bigg) . \]
We will show that
\begin{equation}
\label{eq:Rboundcase2}
\deg R \leq \min\{ 2^{m+n-2} - 2^{m-2} -2, 2^{\ell}-3\},
\end{equation}
from which it follows, by Lemma~\ref{lemma:prodzero}, that $\overline{H}_{2,i}(0)=0$.

Thus, it remains to prove bound~\eqref{eq:Rboundcase2}. Observe that
\begin{align}
\label{eq:Rcase2}
	\deg R & = k + \sum_{s=1}^{t-1}2^{i_s-1} + \sum_{j=m+n-i-1}^{\ell -1} 2^{j-1} \notag \\
	& \leq (2^{m-2} -2 ) + \bigg(\sum_{s=1}^{t-1}2^{i_s-1}\bigg) + \big( 2^{\ell-1} - 2^{m+n-i-2} \big) \\
	&\leq \big( 2^{\ell-1} - 2^{m+n-i-2} - 2\big) + 2^{m-2} + \sum_{j=m-1}^{\ell-1}2^{j-1} \notag \\
	& = \big( 2^{\ell-1} - 2^{m+n-i-2} - 2\big) + 2^{\ell-1}
	\leq 2^\ell -3, \notag
\end{align}
where the second inequality is because $i_1\geq m-1$, and the third is because $i\leq m+n-2$,
and hence $2^{m+n-i-2} \geq 1$.
If $\ell \leq m+n-3$, then because $n\geq 1$, we have
\[ 2^{\ell} -3 \leq 2^{m+n-3} -3 < 2^{m+n-2} - 2^{m-2} -2, \]
so that inequality~\eqref{eq:Rcase2} proves bound~\eqref{eq:Rboundcase2} in that case.

Otherwise, we have $\ell=m+n-2$. Recalling that 
$\{i_1,i_2,\dots,i_t\}$ is a proper subset of $\{m-1,m,\dots,m+n-2\}$, it follows that
\[ \sum_{s=1}^{t-1}2^{i_s-1} \leq \bigg( \sum_{j=m-1}^{\ell-1}2^{j-1} \bigg) - 2^{m-2}
= 2^{\ell-1} - 2^{m-1}. \]
Thus, inequality~\eqref{eq:Rcase2} gives us
\begin{align*}
\deg R & \leq (2^{m-2} -2 ) + \big( 2^{\ell-1} - 2^{m-1} \big) + \big( 2^{\ell-1} - 2^{m+n-i-2} \big) \\
& = 2^\ell - 2^{m-2} - 2^{m+n-i-2}-2 < 2^{m+n-2} - 2^{m-2} - 2,
\end{align*}
completing the proof of bound~\eqref{eq:Rboundcase2} and hence the lemma.
\end{proof}

\begin{lemma}
\label{lemma:multiply by c^k, whole product}
Let $1\leq k\leq 2^{m-2}-2$. Then
\[ T\bigg(c^k\prod_{j=m-1}^{m+n-2}a_j,a_{m+n-1}+a_{m-1}\bigg)
= \bigg( \frac{1}{2^n}-1 \bigg) T\big(c^k, a_{m+n-1}+a_{m-1}\big) .\]
\end{lemma}

\begin{proof}
By Proposition~\ref{prop:contour}, we have 
\[T\bigg(c^k\prod_{j=m-1}^{m+n-2}a_j,a_{m+n-1}+a_{m-1}\bigg)=-\frac{\overline{H}(0)}{2}\]
and
\[T\big(c^k, a_{m+n-1}+a_{m-1}\big) = -\frac{\overline{G}(0)}{2},\]
where
\[ H = c^k(a_{m+n-1}'+a_{m-1}') \prod_{j=m-1}^{m+n-2}a_j
\quad\text{and}\quad
G= c^k\big(a_{m+n-1}'+a_{m-1}' \big) .\]
%	We will study the values $H(0)$ and $G(0)$ separately.

Write $H=A+B$ where
\[ A= c^k a_{m-1}' \prod_{j=m-1}^{m+n-2}a_j
\quad\text{ and }\quad
B = c^k a_{m+n-1}' \prod_{j=m-1}^{m+n-2}a_j .\]
By equation~\eqref{eq:arprime}, we have $a_{r-1} a'_{r-1} = \frac{1}{2} (a'_r - 1)$ for any $r\geq 2$.
Applying this formula inductively, we have
\begin{align*}
A &= c^k\big( a_{m-1}a'_{m-1}\big) \prod_{j=m}^{m+n-2}a_j
= \frac{c^k}{2} a'_m  \prod_{j=m}^{m+n-2}a_j - \frac{c^k}{2} \prod_{j=m}^{m+n-2}a_j \\
& = \cdots = 
\frac{c^k}{2^n} a'_{m+n-1} - \sum_{i=1}^n \frac{c^k}{2^i} \prod_{j=m+i-1}^{m+n-2}a_j
\end{align*}
For each $i=1,\ldots, n$, the $i$-th term in the sum above has degree
\begin{align*}
	\deg \bigg( c^k\prod_{j=m+i-1}^{m+n-2}a_j\bigg)
	& = k + \sum_{j=m+i-1}^{m+n-2}2^{j-1} \leq 2^{m-2} - 2 + 2^{m+n-2} - 2^{m+i-2} \\
	& < 2^{m+n-2}-1 = \deg\big( (a_{m+n-1}+a_{m-1})/c \big) ,
\end{align*}
and since $k\geq 1$, has value $0$ at $c=0$.
Thus, $\overline{A}(0)=2^{-n} \overline{A}_0(0)$, where $A_0=c^k a'_{m+n-1}$.

Next, we turn to $B$. Writing
\begin{align*}
a'_{m+n-1} a_{m+n-2} &= (2a_{m+n-2}a_{m+n-2}'+1) a_{m+n-2}
= 2a_{m+n-2}^2 a_{m+n-2}'+ a_{m+n-2} \\
& = 2 (a_{m+n-1} - c) a_{m+n-2}'+ a_{m+n-2},
\end{align*}
we have $B = B_1 - B_2 + B_3$, where
\[B_1=2 c^k a_{m+n-1} a'_{m+n-2} \prod_{j=m-1}^{m+n-3}a_j,
\quad\quad
B_2 = 2 c^{k+1} a_{m+n-2}' \prod_{j=m-1}^{m+n-3}a_j , \]
and
\[ B_3= c^k \prod_{j=m-1}^{m+n-2}a_j .\]
Since $1\leq k \leq 2^{m-2}-2$, it is straightforward to check that
\[ \deg(B_2), \deg(B_3)< 2^{m+n-2} - 1 = \deg\big( (a_{m+n-1}+a_{m-1})/c\big) ,\]
and that $B_2(0)=B_3(0)=0$. Therefore, because $a_{m+n-1}\equiv -a_{m-1}$
modulo $(a_{m+n-1}+a_{m-1})/c$, we have
$\overline{B}(0) = \overline{B}_1(0) = -\overline C_1(0)$, where
\[ C_1 = 2c^k a'_{m+n-2} a_{m-1} \prod_{j=m-1}^{m+n-3}a_j
= 2c^k a_{m+n-2}a'_{m+n-2} - 2c^{k+1} a'_{m+n-2} \sum_{i=m}^{m+n-2} \prod_{j=i}^{m+n-3} a_j ,\]
where the second equality is by Lemma~\ref{lemma:simpleprod}.

For each $i=m,\ldots, m+n-2$, the $i$-th term in the sum above has degree
\begin{align*}
	\deg \bigg( c^{k+1} a'_{m+n-2} \prod_{j=i}^{m+n-3} a_j \bigg)
	& = k + 2^{m+n-3} + \sum_{j=i}^{m+n-3}2^{j-1}
	= k+2^{m+n-2} - 2^{i-1} \\
	& \leq (2^{m-2} - 2) + 2^{m+n-2} - 2^{m-1} < 2^{m+n-2}-1 \\
	& = \deg\big( (a_{m+n-1}+a_{m-1})/c \big) ,
\end{align*}
and in addition, each such term has value $0$ at $c=0$. Therefore,
$\overline{C}_1(0) = \overline{C}_2(0)$, where
\[ C_2 = 2c^ka_{m+n-2}a_{m+n-2}' = c^k(a_{m+n-1}'-1) =c^ka_{m+n-1}'-c^k . \]
Note that $\deg(c^k)=k < \deg( (a_{m+n-1}+a_{m-1})/c )$
and that $c^k$ has value $0$ at $c=0$, because $k\geq 1$.
Therefore, we have $\overline{C}_2(0)=\overline{A}_0(0)$, where $A_0 = c^ka_{m+n-1}'$ as above.
Combining the above computations, we have
$\overline{H}(0) = \overline{H}_0(0)$, where
\[ H_0 = 2^{-n}A_0 - A_0
%= 2^{-n} c^k a'_{m+n-1} - c^k a_{m+n-1}' 
= \bigg( \frac{1}{2^n} - 1 \bigg) c^k a_{m+n-1}' .\]

On the other hand, we may write $G=G_1 + G_2$ where
\[ G_1 = c^k a_{m-1}' \quad\text{and}\quad G_2 = c^k a_{m+n-1}' = A_0. \]
We have
\[ \deg(G_1) = k+2^{m-2}-1 < 2^{m-2}-1+2^{m-2}-1 = 2^{m-1}-2 < \deg( (a_{m+n-1}+a_{m-1})/c ) \]
and $G_1(0)=0$. Therefore,
\[  \bigg( \frac{1}{2^n} - 1 \bigg) \overline{G}(0) =  \bigg( \frac{1}{2^n} - 1 \bigg) \overline{A}_0 (0)
= \overline{H}_0(0) = \overline{H}(0), \]
completing the proof of Lemma~\ref{lemma:multiply by c^k, whole product}.
\end{proof}

\begin{lemma}
\label{lemma:A_{m,n,k}+B_{m,n,k}}
Let $n\geq 1$, and for each $i=0,\ldots,n-1$, let $k_i\geq 0$.
Define
\[ K=\bigg\lfloor \frac{k_0}{2} \bigg\rfloor + \cdots + \bigg\lfloor \frac{k_{n-1}}{2} \bigg\rfloor .\]
Let $\ell\geq 0$.
Then there are integers $A_0,\ldots,A_{\ell+K}$ and $B_0,\ldots, B_{\ell+K}$ with the following property.
For any $m\geq 2$ such that $\ell + K\leq 2^{m-2}-2$, we have
\[T\bigg( c^{\ell} \prod_{i=0}^{n-1}a_{m+i-1}^{k_i}, a_{m+n-1}+a_{m-1}\bigg) = R+S, \]
where
\[ R=\sum_{j=0}^{\ell+K}A_j T\big(c^j, a_{m+n-1}+a_{m-1}\big) \]
and
\[ S=\sum_{j=0}^{\ell+K}B_jT\bigg(c^j\prod_{i=m-1}^{m+n-2}a_i, a_{m+n-1}+a_{m-1}\bigg). \]
\end{lemma}

\begin{proof}
Let $k=k_0+\cdots+k_{n-1}$; note that $2K\leq k \leq 2K+n$.
Order the set of pairs $(K,k)$ of nonnegative integers with $2K\leq k \leq 2K+n$ as follows:
%(that correspond to a $n$-tuple $k_0,\ldots, k_{n-1}$) as follows:
$(K',k') < (K,k)$ either if $K'<K$, or if $K'=K$ and $k'<k$.
(Note that for any such pair $(K,k)$, there are only finitely many such smaller pairs $(K',k')$.)
We proceed by induction on $(K,k)$, subject to this ordering.

If $k_i=0$ for every $i$, then choose $A_\ell=1$ and all other $A_j,B_j$ to be $0$;
the desired equality is immediate.
Similarly, if $k_i=1$ for every $i$, then choose $B_\ell=1$ and all other $A_j,B_j$ to be $0$.
On the other hand, if $k_i=0$ for some $i$, and $k_i=1$ for other $i$,
then by Lemma~\ref{lemma:multiply by c^k}, the desired quantity is $0$;
setting $A_j=B_j=0$ for all $j$, then, we are done.
The foregoing cases cover all possibilities when $K=0$.
In addition, the values of $A_j$ and $B_j$ are independent of $m$,
subject to the condition $\ell+K \leq 2^{m-2} -2$.

Now consider $k_0,\ldots,k_{n-1}$ for which $K\geq 1$.
Suppose that we already know the result for all pairs $(K',k') < (K,k)$.
Since $K\geq 1$, there is some $0\leq j\leq n-1$ such that $k_j\geq 2$; consider the largest such $j$.
We have
\[ a_{m+j-1}^{k_j} \equiv \begin{cases}
\dsps a_{m+j-1}^{k_j-2} (a_{m+j} - c) & \text{ if } j\leq n-2, \\[1mm]
\dsps a_{m+j-1}^{k_j-2} (-a_{m-1} - c) & \text{ if } j=n-1
\end{cases}
\quad \pmod{ a_{m+n-1}+a_{m-1} } .\]
Thus, we have $c^{\ell} \prod_{i=0}^{n-1}a_{m+i-1}^{k_i} \equiv \pm f_1 - f_2 \pmod{ a_{m+n-1}+a_{m-1} }$, where
\[ f_1 = c^{\ell} \prod_{i=0}^{n-1}a_{m+i-1}^{k'_i} \quad\text{and}\quad
f_2 = c^{\ell+1} \prod_{i=0}^{n-1}a_{m+i-1}^{k''_i}, \]
where the $\pm$ sign is $+$ if $j\leq n-2$, or $-$ if $j=n-1$,
and where
\[ k'_i = \begin{cases}
k_j-2 & \text{ if } i=j \text{ and } n\geq 2, \\
k_i+1 & \text{ if } i\equiv j+1 \, (\text{mod } n) \text{ and } n\geq 2, \\
k_i-1 & \text{ if } i=0 \text{ and } n=1, \\
k_i & \text{ otherwise},
\end{cases}
\quad\text{and}\quad
k''_i = \begin{cases}
k_j-2 & \text{ if } i=j, \\
k_i & \text{ otherwise}.
\end{cases}
\]
Define $K'=\lfloor k'_0/2 \rfloor + \cdots + \lfloor k'_{n-1}/2 \rfloor$
and $K''=\lfloor k''_0/2\rfloor + \cdots + \lfloor k''_{n-1}/2 \rfloor$.
Similarly define $k'=k'_0 + \cdots + k'_{n-1}$ and $k''=k''_0 + \cdots + k''_{n-1}$.
We have $K'\leq K$, and $k' =k-1< k$, and $K''=K-1<K$.
Therefore, $(K',k')<(K,k)$ and $(K'',k'')<(K,k)$.
If we also define $\ell'=\ell$ and $\ell''=\ell+1$,
then $\ell' + K' \leq 2^{m-2} -2$ and $\ell'' + K'' \leq 2^{m-2}-2$.
Hence, by our inductive assumption, the 
desired statement holds for both $f_1$ and $f_2$.
In addition, the resulting integer coefficients $A_j$ and $B_j$,
as linear combinations of the coefficients for $f_1$ and $f_2$,
are independent of $m$, as desired.
\end{proof}

\begin{proof}[Proof of Theorem~\ref{thm:P_{m,n}^2}]
Recall from Newton's identities for symmetric polynomials that in characteristic zero,
the coefficients $c_i$ of a monic polynomial $f(x)=x^N + c_{N-1} x^{N-1} + \cdots + c_0$
can be written in terms of the power sums $p_j=\sum_{i=1}^N \alpha_i^j$
of the roots $\alpha_1,\ldots,\alpha_N$. Moreover, these identities are independent of the degree $N$,
with $c_{N-j}$ depending only on $p_1,\ldots,p_j$;
for example, $c_{N-1}=-p_1$, and $c_{N-2}=\frac{1}{2}(p_1^2 - p_2)$.
Thus, to show that the two polynomials $(P_{m,n})^2$ and $P_{m+1,n}$ have matching
coefficients $b_{k-i}=c_{\ell-i}$ for all $i\geq 0$ up to some bound $B$, it suffices
to show that the power sums of roots of the two polynomials coincide up to the same power $B$.

Recall from equation~\eqref{eq:lambda}
that the roots of $P_{m,n}$ are $-2^n\prod_{j=m-1}^{m+n-2} a_j(\alpha)$, for each root $\alpha$ of $G_{m,n}$;
similarly for $P_{m+1,n}$.
Since $(P_{m,n})^2$ has each root of $P_{m,n}$ appearing twice, then after cancelling $(-2^n)^i$
when equating the relevant power sums,
%in light of Lemma~\ref{lemma:properties of T}(a),
it suffices to show the equality
\begin{equation}
\label{eq:NewtonGoal}
2 T\bigg(\prod_{j=m-1}^{m+n-2}a_j^i, G_{m,n}\bigg) = T\bigg(\prod_{j=m}^{m+n-1}a_j^i, G_{m+1,n} \bigg)
\end{equation}
for any $i$ with $1\leq \lfloor \frac{ni}{2}\rfloor \leq 2^{m-2}-2$.
For the remainder of the proof, fix such $i$.

By Lemma~\ref{lemma:properties of T}(b) and the definition of $G_{m,n}$, for $n\nmid m-1$, we have
\[T\bigg(\prod_{j=m-1}^{m+n-2}a_j^i, G_{m,n} \bigg)
= \sum_{d|n}\mu\bigg(\frac{n}{d}\bigg) T\bigg(\prod_{j=m-1}^{m+n-2}a_j^i, a_{m+d-1}+a_{m-1} \bigg).\]
In fact, by the same reasoning as in the proof of Lemma~\ref{lemma:Mobius sum for trace},
the same equality also holds if $n|m-1$.
Observe that for each $i\geq 0$ and $t\geq 1$, we have
\[ a_{m+dt-1}\equiv -a_{m-1} \quad\text{and}\quad
a_{m+i+dt}\equiv a_{m+i}\pmod{a_{m+d-1}+a_{m-1}}, \]
and hence by Lemma~\ref{lemma:properties of T}(a), we have
\begin{align}
	T\bigg(\prod_{j=m-1}^{m+n-2}a_j^i, G_{m,n} \bigg)
	&= \sum_{d|n}\mu\bigg(\frac{n}{d}\bigg)(-1)^{(n/d)-1} \,
	T\bigg(\prod_{j=m-1}^{m+d-2}a_j^{ni/d},a_{m+d-1}+a_{m-1} \bigg)\notag\\
	\label{eq:Mobius expansion for T}
	&=\sum_{d|n}\mu\bigg(\frac{n}{d}\bigg)(-1)^{(n/d)-1} (R_d + S_d),
\end{align}
where $R_d$ and $S_d$ are the quantities $R$ and $S$ from Lemma~\ref{lemma:A_{m,n,k}+B_{m,n,k}}
for which
\[ T \bigg(\prod_{j=m-1}^{m+d-2}a_j^{ni/d},a_{m+d-1}+a_{m-1} \bigg) = R_d + S_d .\]
(That is, in the notation of Lemma~\ref{lemma:A_{m,n,k}+B_{m,n,k}},
use $d$ in place of $n$, set $k_0=\ldots=k_{d-1}=ni/d$, and use $\ell=0$.
That Lemma does indeed apply, because 
$d\lfloor \frac{ni}{2d}\rfloor\leq \lfloor \frac{ni}{2}\rfloor \leq 2^{m-2}-2$.)
%by the hypotheses of Theorem~\ref{thm:P_{m,n}^2}.

Let $K_d= d\lfloor \frac{ni}{2d}\rfloor \leq 2^{m-2}-2$, and let $A_{d,j},B_{d,j}\in\ZZ$ be the integers
in the sums defining $R_d$ and $S_d$. Then $S_d$ is
\[ \sum_{j=0}^{K_d} B_{d,j} T\bigg(c^j\prod_{s=m-1}^{m+d-2} a_s,a_{m+d-1}+a_{m-1}\bigg)
= \sum_{j=0}^{K_d}B_{d,j} \bigg(\frac{1}{2^d}-1\bigg)T(c^j,a_{m+d-1}+a_{m-1}),\]
by Lemma~\ref{lemma:multiply by c^k, whole product}.
Hence, setting $C_{d,j} = A_{d,j}+(\frac{1}{2^d}-1)B_{d,j}$,
equation~\eqref{eq:Mobius expansion for T} and the formulas defining $R_d$ and $S_d$
from Lemma~\ref{lemma:A_{m,n,k}+B_{m,n,k}} yield
\[ T\bigg(\prod_{j=m-1}^{m+n-2}a_j^i, G_{m,n} \bigg) =
\sum_{d|n}\mu\bigg(\frac{n}{d}\bigg) (-1)^{(n/d)-1}
\sum_{j=0}^{K_d} C_{d,j} T(c^j,a_{m+d-1}+a_{m-1}). \]

Applying Lemmas~\ref{lemma:multiply by c^k, whole product} and~\ref{lemma:A_{m,n,k}+B_{m,n,k}} to
$T(\prod_{j=m}^{m+n-1}a_j^i, G_{m+1,n})$ in the same fashion, we obtain
\[T\bigg(\prod_{j=m}^{m+n-1}a_j^i, G_{m+1,n}\bigg) =
\sum_{d|n}\mu\bigg(\frac{n}{d}\bigg) (-1)^{(n/d)-1}
\sum_{j=0}^{K_d} C_{d,j} T(c^j,a_{m+d}+a_{m}),\]
using the same coefficients $C_{d,j}$, since
the integer coefficients $A_{d,j}$ and $B_{d,j}$ in the sums defining $R_d$ and $S_d$
in Lemma~\ref{lemma:A_{m,n,k}+B_{m,n,k}} are independent of $m$.

Thus, to prove equation~\eqref{eq:NewtonGoal}, it suffices to prove the equalities
\[2T(c^j,a_{m+d-1}+a_{m-1}) = T(c^j,a_{m+d}+a_m)
%\quad\text{for each } j=0,1,\dots, d\bigg\lfloor \frac{ni}{2d}\bigg\rfloor .\]
\quad\text{for each } j=0,1,\ldots, K_d .\]
Equivalently, we must prove
\begin{equation}
\label{eq:NewtonGoal2}
T\big(c^j,(a_{m+d-1}+a_{m-1})^2\big) = T(c^j,a_{m+d}+a_m)
%\quad\text{for each } j=0,1,\dots, d\bigg\lfloor \frac{ni}{2d}\bigg\rfloor .
\quad\text{for each } j=0,1,\ldots, K_d .
\end{equation}
However, for every $d\geq 1$, observe that the coefficients of the $2^{m-2}$ highest-power terms
of the degree-$2^{m+d-1}$ polynomials $(a_{m+d-1}+a_{m-1})^2\in\ZZ[c]$ and $a_{m+d}+a_m\in\ZZ[c]$ coincide.
Indeed, since $a_{s} = a_{s-1}^2+c$, their difference is
\[ (a_{m+d-1}^2 + 2a_{m+d-1} a_{m-1} + a_{m-1}^2) - (a_{m+d-1}^2 + a_{m-1}^2 + 2c)
= 2a_{m+d-1}a_{m-1} - 2c, \]
which has degree $2^{m+d-2} + 2^{m-2}$, which is too small to affect those highest-power terms.
%Therefore, because $d\lfloor\frac{ni}{2d}\rfloor<2^{m-1}-1$,
Therefore, because $K_d \leq 2^{m-2}-2$,
Newton's identities for symmetric polynomials immediately imply
equations~\eqref{eq:NewtonGoal2}, 
and hence equation~\eqref{eq:NewtonGoal} follows.
\end{proof}

\section{Proving multiplier polynomials are $2$-special}
\label{sec:mainproof}

%When $P_{m,n}$ is $2$-special, and $\tr(P_{m,n})$ and $\tr(P_{m+1,n})$
%satisfy a certain upper bound (linear in $m$ and $n$), the following theorem shows
%that $P_{m+1,n}$ must also be $2$-special.
%It will be one of the main tools in the proof of Theorem~\ref{thm:Main result}.

Theorem~\ref{thm:P_{m,n}^2} makes possible the following result,
which will be an essential tool in the proof of Theorem~\ref{thm:Main result}.

\begin{thm}
\label{thm:inductive-2-special}
Let $m\geq 2$, and $n\geq 1$. Suppose that
%v_2(\tr(P_{m,n}))\leq m+\frac{3n}{2} \quad\text{and}\quad
\[ v_2\big(\tr(P_{m+1,n})\big)\leq m+1+\frac{3n}{2}. \]
If $P_{m,n}$ is $2$-special, then $P_{m+1,n}$ is also $2$-special.
\end{thm}

\begin{proof}
Write $(P_{m,n}(x))^2 = x^{k} + \sum_{j=0}^{k-1} b_jx^j$ and $P_{m+1,n}(x) = x^{\ell} + \sum_{j=0}^{\ell-1} c_j x^j$.
%and normalize the $2$-adic valuation $v_2$ so that $v_2(2)=1$.
By Proposition~\ref{prop:prespecial}, we have $v_2(c_{\ell-j}) > nj$ for all $j\geq 1$.
In particular, $v_2(c_{\ell-1})>1$, satisfying the first condition for $P_{m+1,n}$ to be $2$-special.

By \cite[Lemma 6.3]{BG2} and the hypothesis that $P_{m,n}$ is $2$-special,
we have that $P_{m,n}^2$ is $2$-special as well.
In addition, for any index $j\geq 1$ such that $nj < 2^{m-1}-2$,
we have $\lfloor \frac{nj}{2}\rfloor\leq2^{m-2}-2$, and hence $b_{k-j}=c_{\ell-j}$,
by Theorem~\ref{thm:P_{m,n}^2}.
Therefore, if $2\leq j < (2^{m-1}-2)/n$, then
\[ v_2(c_{\ell-j}) = v_2(b_{k-j}) > v_2(b_{k-1}) = v_2(c_{\ell-1}). \]

It remains to prove $v_2(c_{\ell-j}) > v_2(c_{\ell-1})$
for indices $j\geq \max\{ 2 , (2^{m-1}-2)/n \}$.
By hypothesis, we have $v_2(c_{\ell-1}) \leq m+1+\frac{3n}{2}$;
and as noted above, we also have $v_2(c_{\ell-j}) > nj$.
Thus, the desired result follows immediately if $m$ and $n$ satisfy
\begin{equation}
\label{eq:valbound}
m+ 1 + \frac{3n}{2} \leq \max\big\{ 2n, 2^{m-1}-2 \big\} .
\end{equation}
Inequality~\eqref{eq:valbound} holds if and only if
\[ m+1 \leq \frac{n}{2} \quad\text{or}\quad \frac{3n}{2} \leq 2^{m-1} -m - 3, \]
and hence it fails if and only if
\begin{equation}
\label{eq:valfail}
\frac{2}{3} \big( 2^{m-1} - m -3 \big) < n < 2m+2 .
\end{equation}

Observe that
\[ \frac{2}{3} \big( 2^{m-1} - m -3 \big) - (2m+2) = \frac{2}{3} \big( 2^{m-1} - 4m -6 \big), \]
which is negative for $m\geq 6$. Thus, inequality~\eqref{eq:valbound} holds for all pairs $(m,n)$
with $m\geq 6$. Hence, the only pairs $(m,n)$ for which inequality~\eqref{eq:valbound} fails are
those for which $m=2,3,4,5$ and inequality~\eqref{eq:valfail} holds. That is, it remains to consider
$P_{m+1,n}$ for the finitely many cases
\begin{align*}
m=2 \text{ and } & 1 \leq n\leq 5, \quad & m=3 \text{ and } & 1\leq n\leq 7, \\
m=4 \text{ and }  & 1 \leq n\leq 9, \quad & m=5\text{ and } & 6\leq n\leq 11.
\end{align*}
Direct computation with Magma shows that $P_{m+1,n}$ is $2$-special in all 27 of these cases.
(Incidentally, \cite[Theorem 1.7]{BG2} also says that
$P_{m,1}$ and $P_{m,2}$ are $2$-special for all $m\geq 2$,
covering the six cases above with $n=1,2$.)
\end{proof}

We can now prove Theorem~\ref{thm:Main result}.

\begin{proof}[Proof of Theorem~\ref{thm:Main result}]
The statement holds for $p=2$ by \cite[Theorem 1.7]{BG2}, so we can assume $3\leq p \leq 1021$.
By Proposition~\ref{prop:implies} and Theorem~\ref{thm:special},
%By Corollary~\ref{cor:main tool for non-unitness},
it suffices to show that $P_{m,p}$ is $2$-special for any $m\geq 2$.
%and any prime $p$ not exceeding $1024$.
According to Corollary~\ref{cor:trace-doubles}, we have
\begin{equation}
\label{eq:trace-increases-one}
v_2(\tr(P_{m+1,p})) = v_2(\tr(P_{m,p}))+1
\end{equation}
for any $m\geq 11$. Let $\alpha$ be a root of $G_{m,1} = (a_m+a_{m-1})/c$,
which is irreducible over $\QQ$ by \cite[Corollary 1.1]{Gok19}.
Setting $K=\QQ(\alpha)$, Theorem~\ref{thm:Trace(P_{m,p})} gives us
\begin{equation}
\label{eq:trace-multiplier polynomial for p}
\tr(P_{m,p}) = 2^p\Big(\text{Tr}_{K/\mathbb{Q}}\big(a_{m-1}(\alpha)^p\big)+2^{m+p-2}-2^{m-2}\Big).
\end{equation}
%Using part (b) of \cite[Theorem 1.3]{Gok19} for $d=2$, $n=1$, we also obtain
%\[v_2((a_{m-1}(\alpha))^p) = \frac{p}{2^{m-1}-1}.\]
Magma computations of $\text{Tr}_{K/\mathbb{Q}}(a_{m-1}(\alpha)^p)$,
together with equation~\eqref{eq:trace-multiplier polynomial for p}, show that
\begin{equation}
\label{eq:v2bound}
v_2(\tr(P_{m,p}))<m+\frac{3p}{2}
\end{equation}
for all primes $3\leq p\leq 1021$ and all integers $2\leq m\leq 10$.
(See Section~\ref{sec:data} for some of these trace valuations.)
Therefore, applying equation~\eqref{eq:trace-increases-one} inductively, it follows that
the bound~\eqref{eq:v2bound} holds
for all primes $3\leq p\leq 1021$ and \emph{all} integers $m\geq 2$.
By applying Theorem~\ref{thm:inductive-2-special} inductively, then,
it suffices to show that $P_{2,p}$ is $2$-special for every prime $3\leq p\leq 1021$.

By Remark~\ref{remark:m2case}, we have
\[ \tr(P_{2,p}) = 2^{2p} - 2^{p+1}, \]
and hence $v_2(\tr(P_{2,p}))=p+1$.
On the other hand, by Proposition~\ref{prop:prespecial}, writing $P_{2,p}(x) = x^k+\sum_{i=0}^{k-1} b_i x^i$,
we have
\[ v_2(b_{k-j}) > pj \geq 2p > p+1 = v_2(\tr(P_{2,p})) \quad \text{for all } 2\leq j \leq k ,\]
showing that $P_{2,p}$ is $2$-special.
\end{proof}

\section{Appendix: Some trace data for multiplier polynomials}
\label{sec:data}
In the proof of Theorem~\ref{thm:Main result} in Section~\ref{sec:mainproof}, 
we had to rely on Magma computations of $v_2(\tr(P_{m,p}))$, the $2$-adic valuation
of second-from-top coefficient of the multiplier polynomial $P_{m,p}$.
The full set of those valuations, 
for all primes $3\leq p\leq 1021$ and all integers $2\leq m\leq 10$,
would take up too much space here, but we include some of them in Table~\ref{tab:data}.
We also include the values of $m+p$, since we found $v_2(\tr(P_{m,p}))\leq m+p$
in almost all instances, thus certainly satisfying inequality~\eqref{eq:v2bound}.
(But not always! As Table~\ref{tab:data} shows, this sharper bound fails when
$(m,p)$ is (6,23), (7,23) or (10,487), for example. We have marked all such instances in the table
in boldface.)

\begin{table}
\begin{tabular}{|c|c|c||c|c|c|}
\hline
	$(m,p)$ & $v_2(\tr(P_{m,p}))$ & $m+p$ & $(m,p)$ & $v_2(\tr(P_{m,p}))$ & $m+p$ \\
\hline
	$(10,509)$ & $512$ & $519$ &	$(10,503)$ & $511$ & $513$ \\
\hline
	$(10,499)$ & $509$ & $509$ &	$(10,491)$ & $496$ & $501$ \\
\hline
	$(10,487)$ & $500$ & $\textbf{497}$ &	$(10,479)$ & $486$ & $489$ \\
\hline
	$(10,467)$ & $473$ & $477$ &	$(10,463)$ & $470$ & $473$ \\
\hline
	$(10,461)$ & $470$ & $471$ &	$(10,457)$ & $465$ & $467$ \\
\hline
	$(9,251)$ & $256$ & $260$ &	$(9,241)$ & $249$ & $250$ \\
\hline
	$(9,239)$ & $245$ & $248$ &	$(9,233)$ & $240$ & $242$ \\
\hline
	$(9,229)$ & $237$ & $238$ & 	$(9,227)$ & $236$ & $236$ \\
\hline
	$(9,223)$ & $230$ & $232$ &	$(9,211)$ & $219$ & $220$ \\
\hline
	$(9,199)$ & $206$ & $208$ &	$(9,197)$ & $202$ & $206$ \\
\hline
	$(8,113)$ & $120$ & $121$ &	$(8,109)$ & $112$ & $117$ \\
\hline
	$(8,107)$ & $112$ & $115$ &	$(8,103)$ & $109$ & $111$ \\
\hline
	$(8,101)$ & $105$ & $109$ &	$(8,97)$ & $102$ & $105$ \\
\hline
	$(8,89)$ & $96$ & $97$ &		$(8,83)$ & $91$ & $91$ \\
\hline
	$(8,79)$ & $86$ & $87$ &		$(8,73)$ & $78$ & $81$ \\
\hline
	$(7,61)$ & $64$ & $68$ &		$(7,59)$ & $65$ & $66$ \\
\hline
	$(7,53)$ & $56$ & $60$ &		$(7,47)$ & $54$ & $54$ \\
\hline
	$(7,43)$ & $50$ & $50$ &		$(7,41)$ & $45$ & $48$ \\
\hline
	$(7,37)$ & $43$ & $44$ &		$(7,31)$ & $32$ & $38$ \\
\hline
	$(7,29)$ & $33$ & $36$ &		$(7,23)$ & $31$ & \textbf{30} \\
\hline
	$(6,29)$ & $32$ & $35$ &		$(6,23)$ & $30$ & \textbf{29} \\
\hline
	$(6,19)$ & $24$ & $25$ &		$(6,17)$ & $21$ & $23$ \\
\hline
	$(6,13)$ & $19$ & $19$ &		$(6,11)$ & $16$ & $17$ \\
\hline
	$(6,7)$ & $9$ & $13$ &		$(6,5)$ & $9$ & $11$ \\
\hline
	$(6,3)$ & $6$ & $9$ &		$(6,2)$ & $6$ & $8$ \\
\hline
	$(5,13)$ & $18$ & $18$ &		$(5,11)$ & $15$ & $16$ \\
\hline
	$(5,7)$ & $8$ & $12$ &		$(5,5)$ & $8$ & $10$ \\
\hline
	$(5,3)$ & $5$ & $8$ &		$(5,2)$ & $5$ & $7$ \\
\hline
	$(4,5)$ & $7$ & $9$ &		$(4,3)$ & $4$ & $7$ \\
\hline
	$(4,2)$ & $4$ & $6$ &  		$(3,2)$ & $3$ & $5$ \\
\hline
\end{tabular}
\caption{Trace valuations of multiplier polynomials $P_{m,p}$, relative to $m+p$}
\label{tab:data}
\end{table}

\textbf{Acknowledgments}.
The first author gratefully acknowledges the support of NSF grant DMS-2401172.
The authors thank Joe Silverman for his helpful suggestions.
%The authors thank the anonymous referee for their careful reading of the original
%manuscript of the paper, and for their suggestions, which greatly improved the exposition.

\end{document}